\documentclass[12pt]{article}

\usepackage{epsfig,amsmath,amssymb,amsthm,latexsym, subfig}

\def\NN{\mathbb{N}}
\def\RR{\mathbb{R}}

\def\ZZ{\mathbb{Z}}

\def\wphi{\widehat{\phi}}

\def\ta{\tilde{a}}
\def\tb{\tilde{b}}
\def\tA{\tilde{A}}
\def\tB{\tilde{B}}

\def\tgamma{\tilde{\gamma}}

\newcommand{\bfa}{{\bf a}}
\newcommand{\bfb}{{\bf b}}

\newcommand{\bfc}{{\bf c}}

\newcommand{\crossarrow}{{\searrow \kern-1.0em \swarrow}}

\newtheorem{theorem}{Theorem}
\newtheorem{lemma}{Lemma}

\newtheorem{example}{Example}

\begin{document}

\title{Exact regularity of pseudo-splines}

\author{
Michael S. Floater\footnote{
   Centre of Mathematics for Applications, Department of Informatics,
   University of Oslo, PO Box 1053, Blindern, 0316 Oslo, Norway,
   {\it email: michaelf@ifi.uio.no}}
\ and Georg Muntingh\footnote{Centre of Mathematics for
Applications, Department of Mathematics, University of Oslo, PO
Box 1053, Blindern, 0316 Oslo, Norway,
{\it email: georgmu@math.uio.no}}
}
                                                                                
                                                                                
\maketitle

\begin{abstract}
In this paper we review and refine a
technique of Rioul to determine
the H{\"o}lder regularity of a large
class of symmetric subdivision schemes from the spectral radius of a single
matrix. These schemes include those of Dubuc and
Deslauriers, their dual versions, and more generally all the
pseudo-spline and dual pseudo-spline schemes.
We also derive various comparisons between their regularities
using the Fourier transform. In particular we show
that the regularity of the Dubuc-Deslauriers family
increases with the size of the mask.
\end{abstract}

\noindent {\em MSC: 65D10, 26A16}

\noindent {\em Keywords: } Subdivision, H\"older regularity, pseudo-splines.

\section{Introduction}\label{sec:intro}

Subdivision is a recursive method for generating curves, surfaces
and other geometric objects. Rather than having a complete
description of the object of interest at hand, subdivision generates
the object by repeatedly refining its description starting from
a coarse set of control points.
Since subdivision schemes are often easy to implement and very
flexible, they provide a powerful tool for modelling geometry.
However, analyzing their smoothness, or regularity, can be difficult.
The purpose of this paper is to review and refine a method proposed
by Rioul \cite{Rioul} to determine
the H{\"o}lder regularity of a surprisingly large
class of subdivision schemes from the spectral radius of a single
matrix. A joint spectral radius analysis is not required.

Consider the scheme
\begin{equation}\label{eq:scheme}
 f_{j+1,k} = \sum_\ell a_{k-2\ell} f_{j,\ell},
\end{equation}
with finitely supported mask
$\bfa = (a_k)_{k\in \ZZ}$, and coefficients $a_k \in \RR$,
acting on the initial data
$f_{0,k} \in \RR$, $k \in \ZZ$.
At each subdivision level $j \ge 0$,
let $f_j$ be the piecewise linear function
with value $f_{j,k}$ at the point $2^{-j}k$. The
scheme is \emph{convergent} if it has a pointwise limit
$f := \lim_{j\to \infty} f_j$.
We assume that only
a finite number of the initial data $f_{0,k}$ are non-zero, in
which case $f$ has compact support.
In the special case of the cardinal data $f_{0,k} = \delta_{k,0}$,
the support of the limit $f$ is the interval $[K,L]$ if
$a_K,a_L \ne 0$ and $a_k = 0$ for all $k<K$ and $k > L$.
Note that shifting the $a_k$ merely shifts $f$.

The Laurent polynomial
$$
a(z) = \sum_k a_k z^k
$$
is the \emph{symbol} of the scheme.
It is well known \cite{DL},
that a necessary condition for convergence of (\ref{eq:scheme})
is that
$$
 \sum_k a_{2k} = \sum_k a_{2k+1} = 1,
$$
and so we will make this assumption.
This condition can be expressed in terms of the symbol as
\begin{equation}\label{eq:condition2}
 a(-1) = 0 \qquad \hbox{and} \qquad a(1) = 2. 
\end{equation}

Let us now suppose, after shifting the coefficients $a_k$ as necessary,
that $a(z)$ can be factorized as
\begin{equation}\label{eq:SymbolDerived}
a(z) = 2^{-r} (1+z)^{r+1} b(z)
\end{equation}
for some $r \ge 0$, and that
$\bfb = (b_k)_{k\in \ZZ}$,
the mask corresponding to $b(z)$,
is symmetric about $b_0$, i.e., $b_k = b_{-k}$.
Then the Fourier transform of $\bfb$,
$$
 B(\xi) := b(e^{-i\xi}) = \sum_k b_k e^{-ik\xi}, \qquad \xi \in \RR,
$$
is both periodic with period $2\pi$ and real.

Rioul showed in \cite{Rioul} that a lower bound on the 
H{\"o}lder regularity of (\ref{eq:scheme}), defined below,
can be determined from the spectral radius of a \emph{single matrix}
if $B \ge 0$, i.e., if $B(\xi) \ge 0$
for all $\xi \in [-\pi,\pi]$.
A surprisingly large class of schemes are of this type.
For example, one can easily check that they include
all the pseudo-spline schemes, both primal and dual
\cite{DHRS,DS,DS2,DHSS,DDH},
from explicit formulas for $a(z)$.

Rioul further showed that in the special case that
the scheme (\ref{eq:scheme}) is interpolatory, the lower bound is optimal.
We will show more: that the lower bound is optimal whenever
the cardinal function of the scheme
has $\ell^\infty$-stable integer translates.
Using a characterization of such stability due to Jia and Micchelli
\cite{JM}, this leads us to conclude that the lower bound is optimal
under the slightly stricter condition that $B > 0$.
Such schemes include again all the
pseudo-spline and dual pseudo-spline schemes.

We apply these results to compute
and tabulate the regularity of the pseudo-spline schemes,
primal and dual, for low orders.
We then obtain new information about these regularities:
by making pointwise comparisons between the Fourier transforms
of two schemes, we derive inequalities on their regularities.
As an example, we show that the regularity of the Dubuc-Deslauriers
scheme \cite{Dubuc,DD} increases with the size of the mask.

\section{Regularity}\label{sec:reg}

The limit function $f$ has
\emph{H\"older regularity} $\alpha$, $0 < \alpha < 1$,
written $f \in C^\alpha$, if
\[ |f(x) - f(y)| \leq C |x - y|^\alpha \]
for all $x, y \in \RR$, and we write $f \in C^{q+\alpha}$
for $q \in \NN_0$, $0 < \alpha < 1$,
if $f \in C^q$, i.e., $f$ is $q$ times continuously differentiable,
and $f^{(q)} \in C^\alpha$.
Correspondingly, we shall
say that the scheme \eqref{eq:scheme} has
H\"older regularity $\gamma$ for some real $\gamma \ge 0$,
if, for $\beta < \gamma$,
$f \in C^\beta$ for all initial data,
and, for $\beta > \gamma$, $f \not \in C^\beta$
for some initial data.

The regularity of $f$ is related to the behaviour of
the divided differences of the scheme. For each integer $s \ge 0$,
let $f_{j,k}^{[s]}$ denote the divided difference of the values
$f_{j,k-s},\ldots,f_{j,k}$ at the corresponding dyadic points
$2^{-j}(k-s),\ldots,2^{-j}k$. Then
$f_{j,k}^{[0]} = f_{j,k}$ and for $s \ge 1$,
\begin{equation}\label{eq:recurse}
   f_{j,k}^{[s]} = \frac{2^j}{s} (f_{j,k}^{[s-1]} - f_{j,k-1}^{[s-1]}).
\end{equation}
Under conditions (\ref{eq:condition2}) and (\ref{eq:SymbolDerived}),
there is a scheme for the $f_{j,k}^{[s]}$ for $s=0,1,\ldots,r+1$.
For such $s$, if we define the associated Laurent polynomial as
$$
 f_j^{[s]}(z) = \sum_k f_{j,k}^{[s]} z^k,
$$
then
$$
 f_{j+1}^{[s]}(z) = a^{[s]}(z) f_j^{[s]}(z^2),
$$
where
$$ a^{[s]}(z) = \frac{2^s}{(1+z)^s} a(z), $$
from which we obtain the derived scheme
\begin{equation}\label{eq:derivedschemea}
 f_{j+1,k}^{[s]} = \sum_\ell a_{k-2\ell}^{[s]} f_{j,\ell}^{[s]}.
\end{equation}

Then, with
$$ g_{j,k}^{[r]} := f_{j,k}^{[r]} - f_{j,k-1}^{[r]}, $$
it can be shown \cite{DL} that if
\begin{equation}\label{eq:growth}
 |g_{j,k}^{[r]}| \le C\lambda^j,
\end{equation}
for some constants $C$ and $\lambda < 1$, for large enough $j$,
then $f \in C^r$.
Moreover, if $1/2 < \lambda < 1$, then
$f \in C^{r - \log_2(\lambda)}$.

\section{Reduction procedure}\label{sec:reduction}

How can we use (\ref{eq:growth}) in the case that it holds with
$\lambda \ge 1$? Then we do not know whether $f \in C^r$, but if
$r \ge 1$ we
can use the `reduction procedure' of Daubechies,
Guskov, and Sweldens \cite{DGS} to
obtain information about lower order derivatives.
Although the procedure was shown to work for
interpolatory schemes in \cite{DGS},
it also applies to the more general scheme (\ref{eq:scheme}).

\begin{lemma}
Suppose (\ref{eq:SymbolDerived}) holds for some $r \ge 1$.
If (\ref{eq:growth}) holds with $\lambda > 1$ then
$$
 |g_{j,k}^{[r-1]}| \le D_1 2^{-j} \lambda^j,
$$
while if it holds with $\lambda = 1$,
$$
 |g_{j,k}^{[r-1]}| \le (D_2 + D_3j) 2^{-j},
$$
for constants $D_1, D_2, D_3$.
\end{lemma}

\begin{proof}
By the divisibility assumption, $a^{[r]}(-1) = 0$, and
by the assumption that $a(1) = 2$ in
(\ref{eq:condition2}), it also follows that
$a^{[r]}(1) = 2$. Therefore,
$$
 \sum_k a_{2k}^{[r]} = \sum_k a_{2k+1}^{[r]} = 1,
$$
and from \eqref{eq:derivedschemea} using summation by parts
there is a constant $C_1$ such that
$$ |f_{j+1,2k+s}^{[r]} - f_{j,k}^{[r]}| \le
   C_1 \max_k |g_{j,k}^{[r]}|, \qquad s=0,1. $$
So, for any $j \ge 1$, if we represent any $k \in \ZZ$ in binary form
as $k = k_j$, where
$$ k_\ell = 2 k_{\ell-1} + s_\ell, \qquad \ell=j,j-1,\ldots,1, $$
for some $k_0 \in \ZZ$ and $s_1,\ldots,s_j \in \{0,1\}$, then
$$ |f_{j,k}^{[r]} - f_{0,k_0}^{[r]}|
  \le \sum_{\ell=1}^j | f_{\ell,k_\ell}^{[r]} - f_{\ell-1,k_{\ell-1}}^{[r]} |
 \le C_1 C (1 + \lambda + \cdots + \lambda^{j-1}).
$$
Hence,
$$
  |f_{j,k}^{[r]}| \le C_2 + C_1 C (1 + \lambda + \cdots + \lambda^{j-1}),
$$
and since
$$ g_{j,k}^{[r-1]} = 2^{-j} r f_{j,k}^{[r]}, $$
this gives the result in the two cases $\lambda > 1$
and $\lambda = 1$.
\end{proof}

By applying this procedure recursively, it follows that
if (\ref{eq:growth}) holds for any $\lambda$ with
$1/2 < \lambda < 2^r$ then
$f \in C^{r-\log_2\lambda}$ if $\log_2\lambda$ is not an integer,
and
$f \in C^{r-\log_2\lambda - \epsilon}$ for any small $\epsilon > 0$
if $\log_2\lambda$ is an integer.

\section{Rioul's method}\label{sec:rm}

With $r$ in (\ref{eq:SymbolDerived}) now fixed, 
let $g_{j,k} = g_{j,k}^{[r]}$ and
$g_j(z) = \sum_k g_{j,k} z^k$.
Then
\begin{equation}\label{eq:derivedschemeL}
 g_{j+1}(z) = b(z) g_j(z^2),
\end{equation}
with $b(z)$ as in (\ref{eq:SymbolDerived}),
or equivalently,
\begin{equation}\label{eq:derivedscheme}
 g_{j+1,k} = \sum_\ell b_{k-2\ell} g_{j,\ell}.
\end{equation}
Suppose now that $\bfb$ is symmetric, and therefore,
after shifting the coefficients as necessary, it has the form
\begin{equation}\label{eq:bform}
 \bfb = (b_p, \ldots, b_1, b_0, b_1,\ldots, b_p),
  \qquad b_p\neq 0,
\end{equation}
for some $p \ge 0$, in which case
$$
 B(\xi) = b_0 + 2 \sum_{k=1}^{p} b_k \cos(k\xi).
$$
If $p=0$ we must have $b_0 = 1$ and so we can take
$\lambda = 1$. In this case the scheme
(\ref{eq:scheme}) is the B-spline scheme of degree $r$ and
this merely confirms the well-known fact that the limit $f$,
being a spline of degree $r$, belongs to $C^\beta$ for
any $\beta < r$.
Thus, we assume from now on that $p \ge 1$.

Iterating \eqref{eq:derivedschemeL} gives
\begin{equation}\label{eq:Gj}
 g_j(z) = b_j(z) g_0(z^{2^j}),
\end{equation}
where
\begin{equation}\label{eq:bj}
 b_j(z) := b(z) b(z^2) \cdots b(z^{2^{j-1}}).
\end{equation}
But then
\begin{equation}\label{eq:bj1}
 b_{j+1}(z) = b(z) b_j(z^2),
\end{equation}
and so $b_j(z)$ is the Laurent polynomial of
the data $b_{j,k}$, where $b_{0,k} = \delta_{k,0}$ and
\begin{equation}\label{eq:schemeb}
 b_{j+1,k} = \sum_\ell b_{k-2\ell} b_{j,\ell}.
\end{equation}
In particular, $b_{1,k} = b_k$.
Since (\ref{eq:Gj}) can be written as
\begin{equation}\label{eq:gb}
 g_{j,k} = \sum_\ell b_{j,k-2^j\ell} g_{0,\ell},
\end{equation}
it follows that
$$ |g_{j,k}| \le \max_m |b_{j,m}| \sum_\ell |g_{0,\ell}|, $$
and so (\ref{eq:growth}) holds if there is some constant $C'$ such that
\begin{equation*}
 \max_k |b_{j,k}| \le C' \lambda^j.
\end{equation*}

By induction on $j$, the values $b_{j,k}$ are zero whenever
$k < -p_j$ or $k > p_j$, where $p_j := (2^j-1)p$.

\begin{lemma}[Rioul]\label{lem:rioul}
If $b(z)$ in (\ref{eq:SymbolDerived}) has the form (\ref{eq:bform}) and
$B \ge 0$ then
$$ \max_k |b_{j,k}| = b_{j,0} \qquad \text{for all } j \ge 0. $$
\end{lemma}

\begin{proof}
Since
$$ B_j(\xi) := \sum_{k=-p_j}^{p_j} b_{j,k} e^{-ik\xi} $$
is a Fourier series, we have
$$ b_{j,k} = \frac{1}{2\pi} 
     \int_{-\pi}^{\pi} B_j(\xi) e^{ik\xi} \, d\xi, $$
and therefore
$$ |b_{j,k}| \le \frac{1}{2\pi} 
     \int_{-\pi}^{\pi} |B_j(\xi)| \, d\xi. $$
By the symmetry of the $b_k$, it follows from (\ref{eq:schemeb})
by induction on $j$ that
the $b_{j,k}$ are symmetric for all $j$, i.e., $b_{j,-k} = b_{j,k}$.
Therefore, $B_j$ is real,
and by induction on $j$ from (\ref{eq:bj1}), $B_j(\xi) \ge 0$
for all $\xi$, and it follows that
\[ |b_{j,k}| \le \frac{1}{2\pi} 
     \int_{-\pi}^{\pi} B_j(\xi) \, d\xi = b_{j,0}. \qedhere \]
\end{proof}

\section{Spectral radius}\label{sec:spectral}

Under the assumption of Lemma~\ref{lem:rioul} it follows that
(\ref{eq:growth}) holds if
$$
 b_{j,0} \le C \lambda^j
$$
for large enough $j$. One way to determine such $\lambda$
is to study the vector of coefficients
$$ \bfb_j = (b_{j,-p+1},\ldots,b_{j,p-1})^T, $$
since it includes the central coefficient $b_{j,0}$ and is
self-generating in the sense that
$\bfb_{j+1} = M \bfb_j$, where, from (\ref{eq:schemeb}),
$M$ is the matrix of dimension $2p-1$ defined by
\begin{equation}\label{eq:Mlarge}
 M = (b_{k-2\ell})_{k,\ell=-p+1,\ldots,p-1}.
\end{equation}
The first few examples of $M$, with $p=1,2,3$, are
$$
\left[\begin{matrix}
  b_0 
     \end{matrix}\right], \quad
 \left[\begin{matrix}
  b_1    & b_{-1} & 0 \\
  b_2    & b_0 & b_{-2} \\
  0    & b_1 & b_{-1}
     \end{matrix}\right], \quad
 \left[\begin{matrix}
  b_2    & b_0 & b_{-2} & 0      & 0      \\
  b_3    & b_1 & b_{-1} & b_{-3} & 0      \\
  0      & b_2 & b_0    & b_{-2} & 0      \\
  0      & b_3 & b_1    & b_{-1} & b_{-3} \\
  0      & 0   & b_2    & b_0    & b_{-2}
     \end{matrix}\right].
$$

\begin{theorem}\label{thm:MainTheorem}
If $B \ge 0$ then
\begin{equation}\label{eq:bj0lim}
 \lim_{j \to \infty} b_{j,0}^{1/j} = \rho,
\end{equation}
where $\rho$ is the spectral radius of $M$, and if $\rho \ge 1/2$,
a lower bound for the regularity of the scheme \eqref{eq:scheme} is
$r - \log_2(\rho)$.
\end{theorem}

\begin{proof}
Since $\bfb_j = M^j \bfb_0$,
$$ b_{j,0} = \Vert \bfb_j \Vert_\infty
   \le \Vert M^j \Vert_\infty \Vert \bfb_0 \Vert_\infty
   = \Vert M^j \Vert_\infty. $$
On the other hand, from equation (\ref{eq:gb}),
$$ M^j = (b_{j,k-2^j\ell})_{k,\ell=-p+1,\ldots,p-1}, $$
and so
$$ \Vert M^j \Vert_\infty \le (2p-1) \max_k |b_{j,k}|
                          = (2p-1) b_{j,0}. $$
Therefore,
$$ (2p-1)^{-1/j} \Vert M^j \Vert_\infty^{1/j}
    \le b_{j,0}^{1/j} \le \Vert M^j \Vert_\infty^{1/j}, $$
and letting $j\to \infty$ proves (\ref{eq:bj0lim}). 
It follows from (\ref{eq:bj0lim}) that
(\ref{eq:growth}) holds with $C = 1$ for any $\lambda > \rho$,
and this proves the lower bound on the regularity of the scheme.
\end{proof}

\section{Alternative matrices}\label{sec:folded}

Due to the assumption that $\bfb$ is symmetric
we can compute $\rho$ in (\ref{eq:bj0lim})
as the spectral radius of a matrix of roughly half the size of $M$,
namely of dimension~$p$ instead of $2p-1$.
Since $b_{j,-k} = b_{j,k}$, the
vector of coefficients
$$ \bfb_j = (b_{j,0},b_{j,1},\ldots,b_{j,p-1})^T, $$
also includes $b_{j,0}$ and is self-generating.
Indeed, from (\ref{eq:schemeb}),
$$ b_{j+1,k} = b_k b_{j,0} + \sum_{\ell \ge 1} 
         (b_{k-2\ell} + b_{k+2\ell}) b_{j,\ell}, $$
and using the fact that $b_{-k} = b_k$ implies
$$ b_{j+1,k} = b_k b_{j,0} + \sum_{\ell \ge 1} 
         (b_{|k-2\ell|} + b_{k+2\ell}) b_{j,\ell}, $$
and it follows that $\bfb_{j+1} = M \bfb_j$,
where $M$ is the matrix of dimension $p$,
\begin{equation}\label{eq:Msmall}
M = (m_{k,\ell})_{k,\ell=0,\ldots,p-1},
\qquad
m_{k,\ell} = \left\{\begin{array}{ll}
b_k, & \ell = 0;\\
b_{|k-2\ell|} + b_{k+2\ell}, & \ell \ge 1,
\end{array} \right.
\end{equation}
For $p=1,2,3,4$, this `folded' matrix is
$$
\left[\begin{matrix}
  b_0 
     \end{matrix}\right], \quad
 \left[\begin{matrix}
  b_0    & 2b_2 \\
  b_1    & b_1
     \end{matrix}\right], \quad
 \left[\begin{matrix}
    b_0    & 2b_2 & 0      \\
    b_1    & b_1+b_3 & b_3 \\
    b_2    & b_0    & b_2
     \end{matrix}\right], \quad
 \left[\begin{matrix}
    b_0    & 2b_2    & 2b_4 & 0 \\
    b_1    & b_1+b_3 & b_3  & 0 \\
    b_2    & b_0+b_4 & b_2 & b_4 \\
    b_3    & b_1     & b_1 & b_3
     \end{matrix}\right].
$$

Rioul obtained $\rho$ in (\ref{eq:bj0lim})
in two alternative, but equivalent ways,
working from an alternative to (\ref{eq:bj1}).
From (\ref{eq:bj}) there is another recursion:
$$ b_{j+1}(z) = b(z^{2^j}) b_j(z), $$
which gives
$$
 b_{j+1,k} = \sum_\ell b_\ell b_{j,k-2^j\ell}.
$$
This means that the subset of coefficients $b_{j,k}$ whose indices increase
in steps of $2^j$, rather than 1, i.e., $c_{j,k} := b_{j,2^jk}$,
satisfy the recursion
\begin{equation}\label{eq:schemec}
 c_{j+1,k} = \sum_\ell b_\ell c_{j,2k-\ell}
           = \sum_\ell b_{2k-\ell} c_{j,\ell}.
\end{equation}
It follows that the vector
$$ \bfc_j = (c_{j,-p+1},\ldots,c_{j,p-1})^T $$
is self-generating and includes $b_{j,0}$ because
$b_{j,0} = c_{j,0}$, and so we deduce that
$\rho$ in (\ref{eq:bj0lim}) is also the spectral radius of the
matrix $N$ of dimension $2p-1$ where
$\bfc_{j+1} = N \bfc_j$. However, by comparing (\ref{eq:schemec})
with (\ref{eq:schemeb}) we see that $N$ is simply
the transpose of $M$ in (\ref{eq:Mlarge}) and so
these two approaches to computing $\rho$ are equivalent. Rioul computed
$\rho$ in \eqref{eq:bj0lim} from a folded version of
$N$ of dimension $p$, analogous to $M$ in (\ref{eq:Msmall}), using
the reduced vector
$$ \bfc_j = (c_{j,0},\ldots,c_{j,p-1})^T. $$

Theorem~\ref{thm:MainTheorem} holds with $M$ replaced
by each of these alternative matrices, the proof being similar.

\section{Optimality}\label{sec:optimality}

In this section we show that under a slightly stricter condition,
the lower bound on the regularity of Theorem~\ref{thm:MainTheorem}
is optimal.

\begin{theorem}\label{thm:MainTheoremStrict}
If $B > 0$ the lower bound of
Theorem~\ref{thm:MainTheorem} is optimal.
\end{theorem}

To prove this we first establish a lemma that shows that the bound is
optimal whenever the cardinal function of the scheme
has $\ell^\infty$-stable integer translates. The main point
in proving this lemma is that the stability allows us to
bound divided differences of the scheme by corresponding divided
differences of the limit function.

Let $\phi$ denote the cardinal function of the scheme (\ref{eq:scheme}),
i.e., its limit when the initial data is the
cardinal data $f_{0,k} = \delta_{k,0}$.
Then the limit function for general data can be expressed as the
linear combination
$$ f(x) = \sum_\ell f_{0,\ell} \phi(x-\ell). $$
Following Jia and Micchelli \cite{JM},
we shall say that $\phi$
\emph{has $\ell^\infty$-stable integer translates} if there is some
constant $K > 0$ such that
for any sequence $\bfc = (c_\ell)_\ell$ in $\ell^\infty(\ZZ)$,
\begin{equation}\label{eq:stability}
   \Vert \sum_\ell c_\ell \phi(\cdot - \ell) \Vert_{L^\infty(\RR)} 
    \ge K \Vert\bfc\Vert_{\ell^\infty(\ZZ)}.
\end{equation}

\begin{lemma}\label{lem:stable}
Suppose $\phi$ has $\ell^\infty$-stable integer translates and
$f$ has regularity $q + \alpha$ for some $q \in \NN_0$ and $0 < \alpha < 1$.
Then for any integer $r \ge q$, there is a constant $C$ such that
\begin{equation}\label{eq:growthopt}
 |g_{j,k}^{[r]}| \le C 2^{j(r-q-\alpha)}.
\end{equation}
\end{lemma}

\begin{proof}
As is well known, see e.g. the review by Dyn and Levin \cite{DL},
$\phi$ satisfies the two-scale difference
equation
\begin{equation}\label{eq:twoscale}
 \phi(x) = \sum_k a_k \phi(2x-k),
\end{equation}
and therefore, for any $j \ge 0$,
\begin{equation}\label{eq:useful}
 f(x) = \sum_\ell f_{j,\ell} \phi(2^jx-\ell).
\end{equation}
We can use this equation to relate any divided difference
of $f$ of the form
$$ \tilde f_{j,y}^{[q]} := [2^{-j}(y-q),2^{-j}(y-q+1),\ldots,2^{-j}(y)]f, $$
for $y \in \RR$, to the divided differences of the scheme.
Putting $x = 2^{-j}(y-k)$ in (\ref{eq:useful}) gives
$$ f\big(2^{-j}(y-k)\big) = \sum_\ell f_{j,\ell-k} \phi(y-\ell), $$
and, using the cases $k=0,1,\ldots,q$, and
the linearity of divided differences,
$$ \tilde f_{j,y}^{[q]} = \sum_\ell f_{j,\ell}^{[q]} \phi(y-\ell). $$
Similarly, if
$$ \tilde g_{j,y}^{[q]} := \tilde f_{j,y}^{[q]} - \tilde f_{j,y-1}^{[q]}, $$
then
$$ \tilde g_{j,y}^{[q]} = \sum_\ell g_{j,\ell}^{[q]} \phi(y-\ell). $$

Recalling that $f$ has compact support,
if $f$ has regularity $q + \alpha$, there
is some $C>0$ such that for any $\xi_0,\xi_1 \in \RR$,
$$ |f^{(q)}(\xi_1) - f^{(q)}(\xi_0)|
  \le C |\xi_1 - \xi_0|^\alpha, $$
and, by a standard property of divided differences, for each $j$ and $y$,
$$ |\tilde g_{j,y}^{[q]}| = |f^{(q)}(\xi_1) - f^{(q)}(\xi_0)| / q!,$$
for $\xi_0,\xi_1 \in \big(2^{-j}(y-q-1),2^{-j}(y)\big)$.
Therefore, for any $y$,
$$ |\tilde g_{j,y}^{[q]}| \le C' 2^{-j\alpha}, $$
where
$$ C' = C (q+1)^\alpha / q!. $$
Therefore,
$$ \Vert \sum_\ell g_{j,\ell}^{[q]} 
        \phi(\cdot -\ell) \Vert_{L^\infty(\RR)} 
    \le C' 2^{-j\alpha}, $$
and by (\ref{eq:stability}) it follows that for any $\ell \in \ZZ$,
$$
  |g_{j,\ell}^{[q]}| \le K^{-1} C' 2^{-j\alpha}.
$$
Finally, by applying the divided difference definitions
(\ref{eq:recurse}) recursively, $r-q$ times, we obtain (\ref{eq:growthopt}).
\end{proof}

\begin{lemma}\label{lem:stable2}
If $\phi$ has $\ell^\infty$-stable integer translates then
the lower bound, $r - \log_2(\rho)$, of
Theorem~\ref{thm:MainTheorem} is optimal.
\end{lemma}

\begin{proof}
Let $f$ be the limit of the scheme with any initial data for which
$g_{0,k}^{[r]} = \delta_{k,0}$, $-p+1 \le k \le p-1$, and with
only a finite number of initial data $f_{0,k}$ non-zero.
Then $f$ has compact support.
Suppose that $f \in C^{r-\log_2\rho + \epsilon}$ for some
small $\epsilon > 0$ and write the exponent as
$$ r-\log_2\rho + \epsilon = q + \alpha, $$
for $q \in \NN_0$ and $0 < \alpha < 1$. If $\rho > 1/2$, we have
$r \ge q$, and so Lemma~\ref{lem:stable} can be applied, implying
$$ |g_{j,k}^{[r]}| \le C 2^{j(\log_2\rho - \epsilon)}
      = C \rho^j 2^{-j\epsilon}. $$
Thus,
$$ \limsup_{j\to \infty} |g_{j,0}^{[r]}|^{1/j}  \le \rho 2^{-\epsilon}. $$
But $g_{j,0}^{[r]} = b_{j,0}$ of equation (\ref{eq:bj0lim})
and so this contradicts (\ref{eq:bj0lim}).
\end{proof}

Using this lemma we can now prove
Theorem~\ref{thm:MainTheoremStrict} by comparing the cardinal
function $\phi$ with B-splines,
which are known to be stable. A similar idea was used
by Dong and Shen \cite[Lemma 2.2]{DS} to show that
pseudo-splines are stable.

\begin{proof}[Proof of Theorem \ref{thm:MainTheoremStrict}]
By Lemma~\ref{lem:stable2} it is sufficient to show that
$\phi$ has $\ell^\infty$-stable integer translates if $B > 0$.
If the scheme (\ref{eq:scheme}) is interpolatory, in the sense that
$a_{2k} = \delta_{k,0}$, then
$\phi(k) = \delta_{k,0}$ and so
the stability condition (\ref{eq:stability}) holds
with $K = 1$.
To show stability in the general case,
we apply some results by Jia and Micchelli \cite{JM}.
We denote the (continuous) Fourier transform of $\phi$ by
$$ \wphi(\xi) = \int_{\RR} \phi(x) e^{-i\xi x} \, dx,
   \qquad \xi \in \RR. $$
Since the scheme (\ref{eq:scheme}) has constant precision,
$$ \sum_\ell \phi(x-\ell) = 1, \qquad x \in \RR, $$
and, as shown by Jia and Micchelli \cite[Theorem 2.4]{JM}, $\wphi(0) = 1$.
Since the Fourier transform of (\ref{eq:twoscale})
is
$$ \wphi(\xi) = 
   2^{-1} A(\xi/2) \wphi(\xi/2), $$
it follows that
$$ \wphi(\xi) = 
      \prod_{j=1}^\infty \big(2^{-1} A(\xi/2^j)\big). $$
A sufficient condition \cite[Theorem 3.5]{JM} for $\phi$
to have $\ell^\infty$-stable integer translates is that
\begin{equation}\label{eq:sup}
 \sup_{\ell \in \ZZ} |\wphi(\xi + 2 \pi \ell)| > 0,
 \qquad \hbox{for all $\xi \in \RR$.}
\end{equation}
Consider then again
the case that the derived scheme (\ref{eq:derivedscheme})
holds. Then
$$ A(\xi) = 2 \cos^{r+1}(\xi / 2) B(\xi), $$
where, since $A(0) = 2$ under the assumption of convergence,
$B(0) = 1$.
In the B-spline scheme of degree $r$ we have $b(z) = 1$, in which
case we can write the symbol as $a_r(z) = (1+z)^{r+1} / 2^r$.
The cardinal function $\phi_r$ is the B-spline of degree $r$
centred at 0, and we have
$$ \wphi_r(\xi) = 
      \prod_{j=1}^\infty \cos^{r+1}\big(\xi/2^{j+1}\big) = 
   \left(\frac{\sin(\xi/2)}{\xi/2}\right)^{r+1}. $$
It then follows that
$$ \wphi(\xi) = \wphi_r(\xi)
      \prod_{j=1}^\infty B\big(\xi/2^j\big). $$
Since the condition (\ref{eq:sup}) holds for the B-spline $\phi_r$
we deduce that $\phi$ has $\ell^\infty$-stable integer translates
if $B(\xi) > 0$ for all $\xi \in [-\pi, \pi]$.
\end{proof}

\begin{example}
Consider the quintic
Dubuc-Deslauriers scheme \cite{Dubuc,DD} with mask
$$
\bfa = \frac{1}{256}
(3,0,-25,0,150,256,150,0,-25,0,3).
$$
There is a factorization \eqref{eq:SymbolDerived} up
to $r = 5$, in which case one finds
$$
\bfb = (b_{-2},b_{-1},b_0,b_1,b_2)
     = \frac{1}{8} (3,-18,38,-18,3),
$$
with $p=2$, and
$$
 B(\xi)
 = \frac{1}{8} \left(38 - 36 \cos\xi + 6 \cos 2\xi \right).
$$
Making the substitution $s = \sin^2(\xi/2)$ yields
$$ B(\xi) = 1 + 3 s + 6s^2 > 0 $$
for any $\xi \in [-\pi, \pi]$.
Thus
Theorems~\ref{thm:MainTheorem} and~\ref{thm:MainTheoremStrict} both apply.
We find $\rho$ either as the spectral radius of the
matrix (\ref{eq:Mlarge}),
$$
 \frac{1}{8} \left[\begin{matrix}
  -18 & -18 & 0 \\
   3 & 38 & 3 \\
  0 & -18 & -18
     \end{matrix}\right],
$$
or of the smaller, folded matrix
matrix (\ref{eq:Msmall}),
$$
 \frac{1}{8} \left[\begin{matrix}
  38    & 6 \\
  -18    & -18
     \end{matrix}\right].
$$
In both cases we find $\rho = 9/2$ and therefore
the scheme has regularity 
$$ 5 - \log_2(9/2)\approx 2.8301. $$
\end{example}

\section{Pseudo-splines}\label{sec:pseudo}

In the remainder of the paper we focus on the pseudo-spline schemes
and their dual versions, all of which satisfy the conditions of
Theorem~\ref{thm:MainTheorem}
and Theorem~\ref{thm:MainTheoremStrict}.
We first compute numerically their regularities from
the spectral radius of $M$ in (\ref{eq:Msmall}) and
tabulate them.
Then, by making pointwise comparisons among their Fourier
transforms, we derive various comparisons among their
regularities. For example, we show that the regularity
of the Dubuc-Deslauriers family of schemes increases with the polynomial
degree used to define them.

\subsection{Computing regularities}\label{subsec:computing}
For integers $m\geq 1$ and $\ell=1,\ldots,m-1$, the
(primal) pseudo-spline scheme can be defined in terms of its
symbol as
\begin{equation*}
a_{m,\ell}(z) = 2\sigma^m(z) b_{m,\ell}(z), \qquad
b_{m,\ell} (z)
= \sum_{k=0}^\ell \binom{m-1+k}{k} \delta^k(z),
\end{equation*}
where
$$ \sigma(z)=\frac{(1+z)^2}{4z}, \qquad \delta(z) = -\frac{(1-z)^2}{4z}. $$
The Fourier transform of $a_{m,\ell}$ is then
\begin{align*}
A_{m,\ell}(\xi) &= 2\cos^{2m}(\xi/2) B_{m,\ell}(\xi), \cr
B_{m,\ell}(\xi) &= \sum_{k=0}^\ell \binom{m-1+k}{k} \sin^{2k}(\xi/2).
\end{align*}
These schemes can be viewed as a blend between the B-spline and
Dubuc-Deslauriers schemes: when $\ell=0$ the scheme
is B-spline subdivision of degree $2m-1$ and when $\ell = m-1$
the scheme is $(2m)$-point Dubuc-Deslauriers subdivision.
These schemes were introduced by
Daubechies, Han, Ron, and Shen \cite{DHRS}, and
further studied by Dong and Shen in \cite{DS} and \cite{DS2}.

A family of `dual' Dubuc-Deslauriers schemes
was studied by Dyn, Floater, and Hormann \cite{DFH04} and
generalized by Dyn, Hormann, Sabin, and Shen \cite{DHSS} to a
family of dual pseudo-spline schemes defined by the symbol
\begin{equation*}
 \ta_{m,\ell}(z) = \frac{1+z}{z}\sigma^m(z) \tb_{m,\ell}(z), \qquad
 \tb_{m,\ell}(z) = \sum_{k=0}^\ell \binom{m - 1/2 + k}{k} \delta^k(z),
\end{equation*}
for integers $m\geq 1$ and $\ell=1,\ldots,m-1$. The Fourier transform
of $\ta_{m,\ell}$ is
\begin{align*}
\tA_{m,\ell}(\xi)  &= 2e^{i\xi/2}\cos^{2m+1}(\xi/2) \tB_{m,\ell}(\xi), \cr
\tB_{m,\ell}(\xi)  &= \sum_{k=0}^\ell \binom{m-1/2+k}{k} \sin^{2k}(\xi/2).
\end{align*}

Since both $B \ge 1 > 0$ and $\tB \ge 1 > 0$,
Theorems~\ref{thm:MainTheorem} and~\ref{thm:MainTheoremStrict}
apply to both kinds of scheme.

\begin{example}
The primal scheme $a_{4,3}(z)$ is the eight-point Dubuc-Deslauriers
scheme. We can take $r=7$ in the
factorization \eqref{eq:SymbolDerived} and we have
$b(z) = b_{4,3}(z) = b_{-3} z^{-3} + \cdots + b_3 z^3$, with
\[ (b_{-3}, \ldots, b_3) = \frac{1}{16}(-5, 40, -131, 208, -131, 40, -5). \]
So $p=3$ and the folded matrix $M$ in (\ref{eq:Msmall}) has dimension $3$.
Thus $\rho$ is the largest root in absolute value of the cubic polynomial
\[
\det(M - \lambda I) = 
 \det\begin{bmatrix}b_0 - \lambda & 2b_2 & 0\\b_1 & b_1 + b_3 - 
       \lambda & b_3\\b_2 & b_0 & b_2 - \lambda \end{bmatrix} 
  = -\lambda^3 + 7\lambda^2 + \frac{217}{4}\lambda - 125,
\]
which is $\rho \approx 10.91976$ and so
the scheme has regularity $7 - \log_2(\rho) \approx 3.55113$.
\end{example}

Similarly, one can compute the regularities of the schemes
$a_{m,\ell}(z)$ and $\ta_{m,\ell}(z)$ as the
$\log_2$ of algebraic numbers of degree at most $\ell$.
These are shown, to five decimal places,
in Tables \ref{tab:regprimal} and~\ref{tab:regdual}
respectively for $1\le \ell < m \le 8$.
These numbers agree to four decimal places with those computed
from a joint spectral radius in Dong, Dyn, and Hormann \cite{DDH}.

\begin{table}[t]
{\small
\begin{tabular}{crrrrrrr}\hline
      & $l=1$ & $l=2$ & $l=3$ & $l=4$ & $l=5$ & $l=6$ & $l=7$\\
\hline
$m=2$ &  2       &   &   &  &  &  &  \\
$m=3$ &  3.67807 &  2.83007 &   &  &  &  &  \\
$m=4$ &  5.41504 &  4.34379 &  3.55113 &  &  &  &  \\
$m=5$ &  7.19265 &  5.92502 &  4.96207 & 4.19357 &  &  &  \\
$m=6$ &  9       &  7.55781 &  6.43997 & 5.53250 & 4.77675 &  &  \\
$m=7$ & 10.83007 &  9.23111 &  7.97187 & 6.93577 & 6.06273 & 5.31732 &  \\
$m=8$ & 12.67807 & 10.93702 &  9.54804 & 8.39272 & 7.41006 & 6.56398 & 5.82944 \\ \hline
\end{tabular}
}
\caption{Regularities for $a_{m,l}$.}\label{tab:regprimal}
\end{table}

\begin{table}[t]
{\small
\begin{tabular}{crrrrrrr}\hline
      & $l=1$ & $l=2$ & $l=3$ & $l=4$ & $l=5$ & $l=6$ & $l=7$\\
\hline
$m=2$ &   2.83007 &   &   &   &  &  &  \\
$m=3$ &   4.54057 &  3.57723 &   &   &  &  &  \\
$m=4$ &   6.29956 &  5.12711 &  4.24726 &   &  &  &  \\
$m=5$ &   8.09311 &  6.73575 &  5.69355 &  4.85423 &  &  &  \\
$m=6$ &   9.91254 &  8.38994 &  7.19984 &  6.22682 & 5.41143 &  & \\
$m=7$ &  11.75207 & 10.08039 &  8.75493 &  7.65811 & 6.72934 & 5.93283 & \\
$m=8$ &  13.60768 & 11.80033 & 10.35034 &  9.13861 & 8.10385 & 7.20968 & 6.43070 \\ \hline
\end{tabular}
}
\caption{Regularities for $\ta_{m,l}$.}\label{tab:regdual}
\end{table}

\subsection{Comparisons}\label{subsec:comparisons}
In order to make comparisons between the
regularities of the various primal and dual pseudo-spline schemes,
we will show that it is sufficient to make
pointwise comparisons between their corresponding Fourier transforms.
Consider two subdivision schemes defined by
their Fourier transforms $A$ and $\tilde A$,
and suppose that for some integers $r, \tilde r \ge 0$,
\begin{align}
\label{eq:At}
 A(\xi) = 2\cos^{r+1}(\xi/2) B(\xi), \\
\label{eq:tAt}
 \tilde A(\xi) = 2\cos^{\tilde r+1}(\xi/2) \tilde B(\xi),
\end{align}
where $B$ and $\tilde B$ are real and symmetric in $\xi$ and
$B > 0$ and $\tilde B > 0$.
Let $\gamma$ and $\tilde\gamma$ be the respective regularities
of the two schemes.

\begin{lemma}\label{lem:comparison}
If there is a constant $C \geq 1$ such that
$$ \tilde B(\xi) \le C B(\xi), \qquad  \xi \in [-\pi, \pi], $$
then
$$ \tilde\gamma \ge \gamma + \tilde r - r - \log_2 C. $$
\end{lemma}

\begin{proof}
Applying \eqref{eq:bj} twice,
\[ \tilde B_j(\xi) = \tilde B(\xi) \tilde B(2\xi)\cdots \tilde B(2^{j-1}\xi) \le
C^j B_j(\xi), \]
for all $\xi \in [-\pi, \pi]$ and $j \ge 0$. Therefore,
$$ \tilde b_{j,0} = \frac{1}{2\pi} \int_{-\pi}^\pi \tilde B_j(\xi) \, d\xi
   \le C^j \frac{1}{2\pi} \int_{-\pi}^\pi B_j(\xi) \, d\xi
   = C^j b_{j,0},  \qquad j \ge 0, $$
and so
$$ \tilde\rho = \lim_{j \to \infty} (\tilde b_{j,0})^{1/j}
                \le C \lim_{j \to \infty} (b_{j,0})^{1/j} = C \rho, $$
from which the result follows since
\[ \gamma = r+1-\log_2 \rho, \qquad
   \tilde \gamma = \tilde r + 1 - \log_2 \tilde \rho. \qedhere \]
\end{proof}

As an example of the use of this lemma, suppose that $A$ and $\tA$ are
the $2m$- and $2(m+1)$-point Dubuc-Deslauriers schemes
respectively, and that their regularities are $\gamma$ and
$\tilde\gamma$ respectively.
The lemma implies that $\tilde\gamma \ge \gamma$ if
$\tB(\xi) \le 4 B(\xi)$ for all $\xi \in [-\pi, \pi]$.
In turns out that this latter inequality holds.
This is part of the proof of the following more general result.

\begin{theorem}\label{thm:pseudo}
Let $\gamma_{m,\ell}$ be the regularity of the
pseudo-spline scheme defined by $a_{m,\ell}$,
and let $\gamma_m = \gamma_{m,m-1}$. Then
\begin{enumerate}
\item[(i)] $\gamma_{m,\ell}$ is decreasing in $\ell$, and moreover,
$$
   \gamma_{m,\ell-1} - \log_2 \left(\frac{m+\ell}{\ell}\right)
   \le \gamma_{m,\ell} 
   \le \gamma_{m,\ell-1},
$$
\item[(ii)] $\gamma_{m,\ell}$ is increasing in $m$, and moreover,
$$
   \gamma_{m,\ell} + \log_2 \left(\frac{4m}{m+\ell}\right)
   \le \gamma_{m+1,\ell} 
   \le \gamma_{m,\ell} + 2,
$$
\item[(iii)] $\gamma_{m}$ is increasing in $m$, and moreover,
$$
   \gamma_m + \log_2 \left(\frac{2m+2}{2m+1}\right)
   \le \gamma_{m+1} 
   \le \gamma_m + 2.
$$
\end{enumerate}
\end{theorem}

\begin{proof}
Part (i) follows from applying Lemma~\ref{lem:comparison}
with $r=\tilde r=2m$. Since $B_{m,\ell-1}(\xi) \le B_{m,\ell}(\xi)$ for
$\xi \in [-\pi, \pi]$, the lemma implies the second inequality in (i).
To prove the first inequality in (i) we look for a constant $C\ge 1$
such that
\begin{equation}\label{eq:i}
 B_{m,\ell}(\xi) \le C B_{m,\ell-1}(\xi), \qquad \xi \in [-\pi, \pi],
\end{equation}
or equivalently, such that
$$ p(s) := C \sum_{k=0}^{\ell-1} \binom{m-1+k}{k} s^k
           - \sum_{k=0}^\ell \binom{m-1+k}{k} s^k \ge 0,
      \qquad 0 \le s \le 1. $$
Letting
$$ c_k := (C-1) \binom{m-1+k}{k}, \qquad 0 \le k \le \ell-1, $$
we can express $p$ as
$$
 p(s) = \sum_{k=0}^{\ell-1} c_k s^k - \binom{m-1+\ell}{\ell} s^\ell
      = \sum_{k=0}^{\ell-1} c_k (s^k-s^\ell) + c_\ell s^\ell,
$$
where
$$
 c_\ell := \sum_{k=0}^{\ell-1} c_k - \binom{m-1+\ell}{\ell}
     = (C-1) \binom{m+\ell-1}{\ell-1} - \binom{m-1+\ell}{\ell}.
$$
For $s \in [0,1]$, $p(s) \ge 0$ if $c_k \ge 0$, $0 \le k \le \ell$.
Clearly, if $C \ge 1$, $c_k \ge 0$ for $0 \le k \le \ell-1$, and
since
$$
 c_\ell = \frac{(m+\ell-1)!}{\ell!m!} \big((C-1)\ell - m\big),
$$
$c_\ell \ge 0$ if $C \ge (m+\ell)/\ell$.
Thus (\ref{eq:i}) holds with $C = (m+\ell)/\ell$,
and Lemma~\ref{lem:comparison} with this value of $C$
gives the first inequality of (i).

To prove Part (ii), we apply Lemma~\ref{lem:comparison}
with $r=2m$ and $\tilde r = 2m+2$, in which case $\tilde r - r = 2$.
Since $B_{m+1,\ell}(\xi) \ge B_{m,\ell}(\xi)$ for $\xi \in [\pi, \pi]$, the lemma implies the second inequality in (ii).
To prove the first inequality in (ii) we
look for a constant $C \ge 1$ such that
\begin{equation*}
 B_{m+1,\ell}(\xi) \le C B_{m,\ell}(\xi), \qquad \xi \in [-\pi, \pi],
\end{equation*}
or equivalently, such that
$$ p(s) := C \sum_{k=0}^\ell \binom{m-1+k}{k} s^k
           - \sum_{k=0}^\ell \binom{m+k}{k} s^k \ge 0,
      \qquad 0 \le s \le 1. $$
Since
$$
 p(s) = \sum_{k=0}^\ell \frac{(m-1+k)!}{k!m!} \big(Cm - (m+k)\big) s^k,
$$
(\ref{eq:i}) holds with $C = (m+\ell)/m$, and with this $C$,
Lemma~\ref{lem:comparison} implies
the first inequality in (ii).

To prove Part (iii), we again apply Lemma~\ref{lem:comparison}
with $r=2m$ and $\tilde r = 2m+2$, in which case $\tilde r - r = 2$.
Since $B_{m,m-1}(\xi) \le B_{m+1,m}(\xi)$ for $\xi \in [-\pi, \pi]$,
the lemma then implies the second inequality of (iii).
To prove the first inequality we look for a constant $C \geq 1$ such that
\begin{equation}\label{eq:iii}
 B_{m+1,m}(\xi) \le C B_{m,m-1}(\xi), \qquad \xi \in [-\pi, \pi],
\end{equation}
or equivalently, such that
$$ p(s) := C \sum_{k=0}^{m-1} \binom{m-1+k}{k} s^k
           - \sum_{k=0}^m \binom{m+k}{k} s^k \ge 0,
      \qquad 0 \le s \le 1. $$
Letting
$$ c_k := C \binom{m-1+k}{k} - \binom{m+k}{k}, \qquad 0 \le k \le m-1, $$
we can express $p$ as
$$
 p(s) = \sum_{k=0}^{m-1} c_k s^k - \binom{2m}{m} s^m
      = \sum_{k=0}^{m-1} c_k (s^k-s^m) + c_m s^m,
$$
where
$$
 c_m := \sum_{k=0}^{m-1} c_k - \binom{2m}{m}
     = C \binom{2m-1}{m-1} - \binom{2m+1}{m}.
$$
Similar to part (ii), we have
$c_k \ge 0$, $0 \le k \le m-1$, if $C \ge (2m-1)/m$.
On the other hand,
$$ c_m = \frac{(2m-1)!}{(m-1)!(m+1)!} \big(C(m+1) - 2(2m+1)\big), $$
and so $c_m \ge 0$ if $C \ge 2(2m+1)/(m+1)$.
Thus, (\ref{eq:iii}) holds with
$C = 2(2m+1)/(m+1)$, and Lemma~\ref{lem:comparison} then
yields the first inequality of (iii).
\end{proof}

Similar comparisons can be made for the dual
schemes, and also between the primal and dual ones.
To see this observe that Lemma~\ref{lem:comparison}
also holds if $A$ in (\ref{eq:At}) is replaced by
$$
 A(\xi) = 2e^{i\xi/2} \cos^{r+1}(\xi/2) B(\xi),
$$
with $B$ having the same properties as before, and
the lemma also holds with a similar replacement of~$\tilde A$
in (\ref{eq:tAt}).

Consider then the dual schemes.

\begin{theorem}\label{thm:dualpseudo}
Let $\tgamma_{m,\ell}$ be the regularity of the
dual pseudo-spline scheme $\tilde a_{m,\ell}$,
and let $\tgamma_m = \tgamma_{m,m-1}$. Then
\begin{enumerate}
\item[(i)] $\tgamma_{m,\ell}$ is decreasing in $\ell$, and moreover,
$$
   \tgamma_{m,\ell-1} - \log_2 \left(\frac{m+\ell+1/2}{\ell}\right)
   \le \tgamma_{m,\ell} 
   \le \tgamma_{m,\ell-1},
$$
\item[(ii)] $\tgamma_{m,\ell}$ is increasing in $m$, and moreover,
$$
   \tgamma_{m,\ell} + \log_2 \left(\frac{4(m+1/2)}{m+\ell+1/2}\right)
   \le \tgamma_{m+1,\ell} 
   \le \tgamma_{m,\ell} + 2,
$$
\item[(iii)] $\tgamma_{m}$ is increasing in $m$, and moreover,
$$
   \tgamma_m + \log_2 \left( \frac{4m(m+3/2)}{(2m+1/2)(2m+3/2)} \right)
   \le \tgamma_{m+1} 
   \le \tgamma_m + 2.
$$
\end{enumerate}
\end{theorem}

\begin{proof}
The proof of Part (i) is similar to that of
Theorem~\ref{thm:pseudo} but with $C$ replaced by $(m+\ell+1/2)/\ell$.
Part (ii) is also similar to that of
Theorem~\ref{thm:pseudo} but with $C$ replaced by $(m+\ell+1/2)/(m+1/2)$.
Part (iii) is again similar, but we now have $c_k \ge 0$ for
$0 \le k \le m-1$ if $C \ge (2m-1/2)/(m+1/2)$, and
$$ c_m = C \binom{2m-1/2}{m-1} - \binom{2m+3/2}{m} \ge 0 $$
if $C \ge (2m+1/2)(2m+3/2)/\big(m(m+3/2)\big)$.
\end{proof}

Finally, we compare the regularities of the
primal and dual pseudo-splines.

\begin{theorem}\label{thm:primaldual}
For $m \ge 1$ and $0 \le \ell \le m-1$,
\begin{align}
   \gamma_{m,\ell} + \log_2 \left(2 \prod_{n=0}^{l-1} \frac{m + n}{m + 1/2 + n}\right)
   & \le \tgamma_{m,\ell} 
   \le \gamma_{m,\ell} + 1,\label{eq:primaldual1}\\
   \tgamma_{m,\ell} + \log_2 \left(2 \prod_{n = 0}^{l-1} \frac{m + 1/2 + n}{m + 1 + n} \right)
   & \le \gamma_{m+1,\ell} 
   \le \tgamma_{m,\ell} + 1,\label{eq:primaldual2}
\end{align}
where an empty product is understood to mean 1.
\end{theorem}

\begin{proof}
We only prove \eqref{eq:primaldual1}, since the proof of \eqref{eq:primaldual2} is similar.
We apply Lemma~\ref{lem:comparison}
with $r=2m-1$ and $\tilde r = 2m$, in which case $\tilde r - r = 1$.
Since $\tilde B_{m,\ell}(\xi) \ge B_{m,\ell}(\xi)$ for
$\xi \in [-\pi, \pi]$, the lemma implies the second inequality
in \eqref{eq:primaldual1}.
To prove the first inequality in \eqref{eq:primaldual1}
we look for a constant $C \ge 1$ such that
\begin{equation}\label{eq:pd}
 \tilde B_{m,\ell}(\xi) \le C B_{m,\ell}(\xi), \qquad \xi \in [-\pi, \pi],
\end{equation}
or equivalently, such that
$$ p(s) := \sum_{k=0}^\ell c_k s^k \ge 0,
      \qquad 0 \le s \le 1, $$
where
$$ c_k = C \binom{m-1+k}{k} - \binom{m-1/2+k}{k}. $$
For any $k = 0, \ldots, l$, one has $c_k \ge 0$ if and only if
\[ C \geq \left.\binom{m-1/2+k}{k} \right/\binom{m-1+k}{k} = \prod_{n=0}^{k-1} \frac{m + 1/2 + n}{m + n}. \]
So \eqref{eq:pd} holds if we take
\[ C = \prod_{n=0}^{l-1} \frac{m + 1/2 + n}{m + n},\]
and applying Lemma~\ref{lem:comparison} gives the first inequality
in \eqref{eq:primaldual1}.
\end{proof}

\section*{Acknowledgments}
We wish to thank Maria Charina and Nira Dyn for helpful discussions.

\bibliographystyle{plain}

\end{document}